\documentclass[a4paper,11pt]{article}
  \usepackage{amsmath,amsthm,amsfonts,amssymb}
  \theoremstyle{plain}
  \newtheorem{thm}{Theorem}[section]
  \newtheorem{lem}[thm]{Lemma}

  \newtheorem{que}[thm]{Question}
  \newtheorem*{thmA}{Theorem A}
  \newtheorem*{thmB}{Theorem B}
  \newtheorem*{thmC}{Theorem C}

\newcommand{\Irr}{{\operatorname{Irr}}}
\newcommand{\SL}{{\operatorname{SL}}}

\newcommand{\Ker}{{\operatorname{Ker}}}

  \title{Conjugacy classes, characters and products of elements}
  \author{}
  \date{}

\begin{document}

  \maketitle

  \bigskip
  \centerline{by}
  \bigskip

\centerline{Robert M. Guralnick}
\centerline{Department of Mathematics}
\centerline{University of Southern California}
\centerline{Los Angeles,  CA 90089-2532 USA }
\centerline{E-mail: guralnic@usc.edu}
\smallskip
\centerline{and}

  \smallskip
  \centerline{Alexander Moret\'o} \centerline{Departament
  d'Matematiques} \centerline{Universitat de Val\`encia}
  \centerline{46100 Burjassot. Val\`encia SPAIN} \centerline{E-mail:
  Alexander.Moreto@uv.es}

\vskip 10pt

\vskip 10pt

{\bf Abstract.} Recently, Baumslag and Wiegold proved that a finite group $G$ is nilpotent if and only if $o(xy)=o(x)o(y)$ for every $x,y\in G$ of  coprime order. Motivated by this result, we study the groups with the property that $(xy)^G=x^Gy^G$ and those with the property that $\chi(xy)=\chi(x)\chi(y)$ for every $\chi\in\Irr(G)$ and every nontrivial $x, y \in G$ of pairwise coprime order. We also consider several ways of weakening the hypothesis on $x$ and $y$. While the result of Baumslag and Wiegold is completely elementary, some of our arguments here  depend on (parts of) the classification of finite simple groups.

  \vfill

  \noindent The research of the second author is  partially supported by Prometeo II/Generalitat Valenciana, MTM2016-76196-P and FEDER funds.
  \newpage

  \section{Introduction}

Recently,  B. Baumslag and J. Wiegold \cite{bw} obtained the following elementary characterization of finite nilpotent groups: a finite group $G$ is nilpotent if and only if $o(xy)=o(x)o(y)$ for every $x,y\in G$ with $(o(x),o(y))=1$. In \cite{ms}, the second author and A. S\'aez showed that it suffices to assume that $o(xy)=o(x)o(y)$ for $x, y\in G$ {\it of prime power order} with $(o(x),o(y))=1$. The goal of this note is to study some variations of this result by considering conjugacy classes or characters instead of element orders. In Section 2, we consider conjugacy classes. Our first result is elementary.

\begin{thmA}
Let $G$ be a finite group. Then $G$ is nilpotent if and only if $|(xy)^G|=|x^G||y^G|$ for every $x,y\in G$ of prime power order with $(o(x),o(y))=1$.
\end{thmA}

In the next result, we weaken the condition $|(xy)^G|=|x^G||y^G|$.

\begin{thmB}
Let $G$ be a finite group and assume that $(xy)^G=x^Gy^G$ for every $x,y\in G$ of prime power order with $(o(x),o(y))=1$. Then $G$ is solvable. 
\end{thmB}

Theorem B is a consequence of a recent result of the first author and P. H. Tiep that relies on J. Thompson's classification of minimal simple groups. Unfortunately, we have not been able to classify the (solvable) groups that satisfy the hypothesis of Theorem B.  

A related condition was studied by E. Dade and M. Yadav in \cite{dy}. In that article, they classified the groups that satisfy that $x^Gy^G=(xy)^G$ for every $x,y\in G$ with $x^G\neq(y^{-1})^G$. It turns out that such groups do not need to be nilpotent either. However, there are not many nonnilpotent groups satisfying the property studied by Dade and Yadav. They showed that the Frobenius   group $Q_8(C_3\times C_3)$ and $F^+\rtimes F^x$, where $F$ is a finite field with $|F|>2$ are the only nonnilpotent groups with this property.

In Section 3, we classify the finite groups that satisfy a similar condition for characters. This result does not depend on any part of the classification of finite simple groups.

\begin{thmC}
Let $G$ be a finite group. The property $\chi(xy)=\chi(x)\chi(y)$ for every nonidentity elements $x,y\in G$ of prime power order with $(o(x),o(y))=1$ and every $\chi\in\Irr(G)$ holds if and only if either $G$ is a prime power order group, an abelian group or $G$ has the following structure:
\begin{enumerate}
\item
$G=AP$ where $A>1$ is a cyclic $p'$-group and $P$ is a normal Sylow $p$-subgroup for some prime $p$.
\item
If $C=C_P(A)$ and $D=[A,P]$, then $P=CD$ with $C$ abelian, $D\triangleleft G$ and $C\cap D=1$. 
\item
The action of $A$ on $D$ is  Frobenius.
\end{enumerate}
\end{thmC}

One could think that, perhaps, the groups described in (i)-(iii) satisfy that $C$ is normal in $G$ (and hence $G$ can be written as the direct product of an abelian group and a Frobenius group). However, this is not the case. It is well-known that inside the semidirect product $\SL(2,3)(C_3\times C_3)$ the action of $Q_8\leq \SL(2,3)$ on $C_3\times C_3$ is Frobenius. In particular, if we take $z\in Z(Q_8)$ and $g\in\SL(2,3)$ of order $3$, then the semidirect product $G=\langle gz\rangle(C_3\times C_3)$ has the structure described in (i)-(iii) with $C=\langle g\rangle$ not normal in $G$. 

We also consider what happens if we fix a prime  $p$ and we just assume that the class condition of Theorems A or   B or the character condition of Theorem C holds when $x$ is a $p$-element and $y$ is a $p'$-element of prime power order. We obtain similar results, but the strong form of Theorem B relies on the classification of finite simple groups.  We conclude in Section 4 with some open questions motivated by this work.

\section{Conjugacy classes}

We begin with the elementary proof of (a strong form) of Theorem A. It is well-known that a finite group is nilpotent if and only if every two elements of coprime order commute (see Theorem 2.12 of Chapter 4 of \cite{suz}). We need the following slight generalization of this result.

\begin{lem}
\label{com}
Let $p$ be a prime number and $G$ a finite group. A Sylow $p$-subgroup of $G$ is a direct factor of $G$ if and only if for every $p$-element $x$ and every $p'$-element of prime power order $y$, $x$ and $y$ commute. In particular, $G$ is nilpotent if and only if for every prime power order elements $x,y\in G$ with $(o(x),o(y))=1$, $x$ and $y$ commute.
\end{lem}

\begin{proof}
Assume that every $p$-element $x$ commutes with every $p'$-element of prime power order $y$. We argue by induction on $|G|$. We may assume that $p$ divides $|G|$. Let $1\neq z\in Z(P)$, where $P$ is a Sylow $p$-subgroup of $G$. By hypothesis, $C_G(z)$ contains a Sylow $r$-subgroup of $G$ for every prime $r$, so $z\in Z(G)$ so $z\in Z(G)$. By induction $G/\langle z\rangle$ has a Sylow $p$-subgroup as a direct factor whence so does $G$. 
\end{proof}

%\begin{proof}
%Assume that every two elements of prime power order commute. We argue by induction on $|G|$. Let $p$ be a prime divisor of $|G|$ and $1\neq z\in Z(P)$, %where $P$ is a Sylow $p$-subgroup of $G$. By hypothesis, $C_G(z)$ contains a Sylow $r$-subgroup of $G$ for every prime $r$ whence $z\in Z(G)$ and %$Z(G)\neq1$. By induction, $G/Z(G)$ is nilpotent and hence $G$ is nilpotent.
%\end{proof}

Actually, a weaker sufficient condition for nilpotence is known: $G$ is nilpotent if and only if for every pair of primes $p$ and $q$ and every pair of elements $x,y \in G$ with $x$ a $p$-element and $y$ a $q$-element, $x$ commutes with some conjugate of $y$ (see Corollary E of \cite{dghp}), but the proof of this result depends on the classification of finite simple groups

%\begin{proof}
%Assume that every two elements of prime power order commute. Let $x,y\in G$ with $(o(x),o(y))=1$. By the result mentioned above, we want to see that $x$ and %$y$ commute. Write $x=x_{p_1}\dots x_{p_n}$ be the factorization of $x$ as a product of pairwise commuting prime power order elements (see Lemma 8.18 of %\cite{isa}) and $y=y_{q_1}\dots y_{q_m}$ the factorization of $y$. Since $(o(x),o(y))=1$, $p_i\neq q_j$ for every $i,j$. By hypothesis, for every $j$, $y_{q_j}$ %commutes with every $x_{p_i}$. It follows that $x$ and $y$ commute.
%\end{proof}

Now, we are ready to prove our strong form of Theorem A.

\begin{thm}
\label{a}
Let $p$ be a prime number and $G$  a finite group. Then $G$ has a Sylow $p$-subgroup as a direct factor if and only if $|(xy)^G|=|x^G||y^G|$ for every $p$-element $x\in G$ and every $p'$-element of prime power order $y\in G$. In particular, 
 $G$ is nilpotent if and only if $|(xy)^G|=|x^G||y^G|$ for every $x,y\in G$ of prime power order with $(o(x),o(y))=1$.
\end{thm}

\begin{proof}
Assume that $|(xy)^G|=|x^G||y^G|$ for every $p$-element $x\in G$ and every $p'$-element of prime power order $y\in G$.  By the previous lemma, it suffices to see that $x$ and $y$ commute. We have that 
$$
|G:C_G(xy)|=|G:C_G(x)||G:C_G(y)|.
$$
Therefore, 
$$
|G:C_G(x)\cap C_G(y)|\leq|G:C_G(x)||G:C_G(y)|=|G:C_G(xy)|.
$$ Since $C_G(x)\cap C_G(y)\leq C_G(xy)$, we deduce that $C_G(xy)=C_G(x)\cap C_G(y)$. It follows that $xy\in C_G(xy)\leq C_G(y)$, so $xy$ commutes with $y$.  We deduce that $x$ and $y$ commute.
\end{proof}

Next, we notice that the hypothesis in Theorem B is equivalent to a character-theoretic condition.

\begin{lem}
\label{ch}
Let $G$ be a finite group and let $x,y\in G$. Then $(xy)^G=x^Gy^G$ 
  if and only if $\chi(1)\chi(x^ay^b)=\chi(x)\chi(y)$ for every $a, b\in G$ and  $\chi\in\Irr(G)$
\end{lem}

\begin{proof}
The ``only if" part follows from Lemma 2.2 of \cite{gmt}. In order to obtain the `if" part we take $z=x^ay^b\in x^Gy^G$ and we want to see that $z\in(xy)^G$. 
We have that 
$$
\chi(1)\chi(z)=\chi(x^a)\chi(y^b)=\chi(x)\chi(y)=\chi(1)\chi(xy)
$$
for any $\chi\in\Irr(G)$ and we deduce that $z\in(xy)^G$.
\end{proof}

%Arguing as in the proof of this lemma, we also get the following result, which will be relevant in Section 3.

%\begin{cor}
%\label{charcond}
%Let $G$ be a finite group and let $x,y\in G$. 
 % If $\chi(xy)=\chi(x)\chi(y)$ for every $\chi\in\Irr(G)$ then $(xy)^G=x^Gy^G$.
% \end{cor}

The next result depends on the classification of finite simple groups.

\begin{thm}
\label{gt}
Let $G$ be a finite group and let $p$ be a prime number. Then $G$ is $p$-solvable if and only if for every primes $q\neq r$ different from $p$ and every $x\in G$ nontrivial $p$-element, $y\in G$ nontrivial $q$-element, $z\in G$ nontrivial $r$-element, $xyz\neq 1$. 
\end{thm}

\begin{proof}
This follows from Theorems 3.4 and 1.4 of \cite{gt}.
\end{proof}

We note that the case $p=2$ depends on the classification of minimal simple groups, but when $p$ is odd we need the full classification of finite simple groups.

Now, we are ready to prove a strong form of Theorem B.

\begin{thm}
\label{st}
Let $G$ be a finite group and let $p$ be a prime number. If $x^Gy^G=(xy)^G$ for every $p$-element $x$ and every $p'$-element  of prime power order $y$, then $G$ is $p$-solvable. 
\end{thm}

\begin{proof}
Let $G$ be a minimal counterexample. Since the hypothesis is inherited by quotients, any minimal normal subgroup of $G$ is a direct product of nonabelian simple groups of order divisible by $p$. If $N=S_1\times\cdots\times S_t$ is a minimal normal subgroup then, by the minimality of $G$, $G/N$ is $p$-solvable, and we deduce that $N$ is the unique minimal normal subgroup of $G$. 

By Theorem \ref{gt}, there exist primes $q\neq r$ different from $p$ and there exist a nontrivial $p$-element $x\in S_1$, a nontrivial $q$-element $y\in S_1$ and a nontrivial $r$-element $z\in S_1$ such that $xyz=1$. By Lemma \ref{ch}, we have that 
$$
\chi(x)\chi(y)=\chi(1)\chi(xy)=\chi(1)\chi(z^{-1})
$$
and
$$
\overline{\chi(x})\overline{\chi(z)}=\chi(x^{-1})\chi(z^{-1})=\chi(1)\chi(x^{-1}z^{-1})=\chi(1)\chi(y)
$$
for every $\chi\in\Irr(G)$. We deduce that $\chi(y)\neq0$ if and only if $\chi(z)\neq 0$ and in those cases
$|\chi(x)|=\chi(1)$ whence $x\in Z(\chi)$. Since the normal closure of $x$ is $N$, we deduce that $N\leq Z(\chi)$. Since $Z(\chi)/\Ker(\chi)$ is abelian, it follows that $N\leq \Ker\chi$ so $\chi$ is a character of $G/N$. It follows that 
$$
|C_G(y)|=\sum_{\chi\in\Irr(G)}|\chi(y)|^2=\sum_{\chi\in\Irr(G/N)}|\chi(y)|^2=\sum_{\chi\in\Irr(G/N)}|\chi(1)^2|=|G/N|
$$
so $y^G=N$ which is a contradiction (since $1\in N$).
\end{proof}

We observe that Theorem B follows immediately from the case $p=2$ of Theorem \ref{st}

%\begin{proof}
%Let $G$ be a minimal counterexample. By Theorem \ref{gt}, there exist primes $q\neq r$ different from $p$ and  $x\in G$ nontrivial $p$-element, $y\in G$ %nontrivial $q$-element, $z\in G$ nontrivial $r$-element such that  $xyz=1$. Let $H=\langle x, y, z\rangle$. Using Theorem \ref{gt} again, we have that $H$ is also %a counterexample, so $G=H$. Note that this implies that $G$ is perfect (in any abelian quotient, $xyz=1$ implies that $x=y=z=1$). 

%By Lemma \ref{ch}, $\chi(1)\chi(z^{-1})\chi(1)\chi(xy)=\chi(x)\chi(y)$ for every $\chi\in\Irr(G)$. Similarly, $\chi(z)\chi(y)=\chi(1)\chi(x^{-1})$. In particular, if %$\chi(y)=0$ then $\chi(x)=\chi(z)=0$. In fact, if one of $\chi(x), \chi(y)$ or $\chi(z)$ is $0$ then all of them are $0$.

%If they are all nonzero, then multiplying the equations together gives $|\chi(x)|=|\chi(z)|$ and similarly, $|\chi(y)|=|\chi(z)|$. Since $\chi(1)\chi(z)=\chi(x)\chi(y)$, we %deduce that $|\chi(x)|=|\chi(y)|=|\chi(z)|=\chi(1)$, i.e.,  $x, y$ and $z$ are all central modulo the kernel of $\chi$. Since $G=\langle x,y z\rangle$, we deduce that %$G/\Ker\chi$ is abelian whence $\chi$ is the trivial character. Thus $\chi(x)=0$ for every nontrivial irreducible character $\chi$. This is obviously impossible as it %implies $|C_G(x)|=\sum|\chi(x)|^2=1$.
%\end{proof}

\section{Characters}

%Now, we start working toward a proof of Theorem C.  We will use the following (probably known) result.

%\begin{lem}
%\label{van}
%Let $P$ be a normal Sylow $p$-subgroup of a finite group $G$.  Assume that $\chi\in\Irr(G)$ vanishes on $G-P$. Then $\chi(1)_{p'}=|G/P|$ and $\chi$ is induced %from some irreducible character of $P$.
%\end{lem}

%\begin{proof}
%If $q\neq p$ is prime, then $\chi$ vanishes on all $q$-singular  elements, whence $\chi$ has $q$-defect zero (see the proof of Theorem 3.18 of \cite{navb}, for %instance). We conclude that $\chi(1)_{p'}=|G/P|$, as desired. The second assertion follows immediately from Cliford theory.
%\end{proof}

We start working toward a proof of a strong form of Theorem C. Part (iii) of the following lemma will be the key of our proof.

\begin{lem}
\label{key}
Let $G$ be a finite group and let $A$ and $B$ be nontrivial conjugacy classes.  Assume that  $\chi(xy)=\chi(x)\chi(y)$ for every  $x\in A$, $y\in B$ and $\chi\in\Irr(G)$. Then
\begin{enumerate}
\item
$(xy)^G=AB$.
\item
$\chi(xy)=0$ for every nonlinear $\chi\in\Irr(G)$. In particular, $\chi(x)=0$ or $\chi(y)=0$.
\item
$(xy)^G=xyG'$.
\end{enumerate}
\end{lem}

\begin{proof}
In order to obtain (i) it is enough to argue as in the proof of Lemma \ref{ch}. By Lemma \ref{ch}, (i)  implies that $\chi(1)\chi(xy)=\chi(x)\chi(y)$. Since $\chi(xy)=\chi(x)\chi(y)$ by hypothesis, we deduce that $\chi(xy)=0$ if $\chi\in\Irr(G)$ is nonlinear and hence we have (ii). This implies that 
$$
|C_G(xy)|=\sum_{\chi\in\Irr(G)}|\chi(xy)|^2=|G:G'|,
$$
so $(xy)^G=xyG'$. This completes the proof.
\end{proof}

Notice that Lemma \ref{key} (iii) implies that $G$ is solvable, by Theorem 4.3 of \cite{lad} or Theorem B (a) of \cite{gn}. The proof of these results depends on the classification of finite simple groups, but we will not use this in the proof of Theorem \ref{stchar}. It is worth pointing out that the fact that a group $G$ with a conjugacy class of size $|G'|$ is the dual for conjugacy classes of the Howlett-Isaacs theorem on groups of central type (see \cite{hi}).

\begin{thm}
\label{stchar}
Let $G$ be a finite group and let $p$ be a prime. The property $\chi(xy)=\chi(x)\chi(y)$ for every nonidentity  $p$-element $x\in G$, every nonidentity $p'$-element of prime power order $y\in G$  and every $\chi\in\Irr(G)$ holds if and only if $G$ is either a $p$-group,  a $p'$-group,  an abelian group or $G$ has the following structure:
\begin{enumerate}
\item
$G=AH$ where $A>1$ is  cyclic and $H$ is either a normal Sylow $p$-subgroup of $G$ or a normal Hall $p'$-subgroup of $G$.
\item
If $C=C_H(A)$ and $D=[H,A]$, then $H=CD$ with $C$ abelian, $D\triangleleft G$ and $C\cap D=1$. 
\item
The action of $A$ on $D$ is  Frobenius.
\end{enumerate}
\end{thm}

\begin{proof}
Assume that $\chi(xy)=\chi(x)\chi(y)$ for every nonidentity  $p$-element $x\in G$, every nonidentity $p'$-element of prime power order $y\in G$  and every $\chi\in\Irr(G)$.
We may assume that $p$ divides $|G|$ and that $G$ is not a $p$-group. Also, we may assume that $G$ is not abelian.
By Lemma \ref{key}, $(xy)^G=xyG'$ for every nonidentity $p$-element $x\in G$ and every nonidentity $p'$-element of prime power order $y$. 

We claim that $G'$ is either a $p$-group or a $p'$-group. By way of contradiction, assume that there exists a nonidentity $p$-element $x'\in G'$ and a nonidentity $q$-element $y'\in G'$, for some prime $q\neq p$. Then $(x'y')^G=x'y'G'=G'$ and this is a contradiction (since $1\in G'$). 

Assume first that $G'$ is a $p'$-group. Then $G$ has a normal Hall $p'$-subgroup $H$, $G'\leq H$ and $G=AH$ where $A$ is a Sylow $p$-subgroup of $G$. By coprime action, $H=CD$ where $C=C_H(A)$ and $D=[H,A]=G'$. Furthermore, $A$ is abelian, $D\triangleleft G$ and $C\cap D=1$. Let $1\neq a\in A$ and $g\in G'$ be a nonidentity prime power order element. Then $(ag)^G=agG'=aG'$ is a conjugacy class so $a^G=aG'$. This implies that $|C_G(a)|=|G:G'|=|AC|$. Since $AC\leq C_G(a)$, we deduce that $C_G(a)=AC$ and $C_D(a)=1$, so the action of $A$ on $D$ is Frobenius. In particular, since $A$ is abelian, it is cyclic. 
We are done in this case.

Finally, we may assume that $G'$ is a $p$-group. This case can be handled in the same way. The only difference is that $H$ will be a Sylow $p$-subgroup and $A$ a Hall $p'$-subgroup.

Conversely, we want to see that if a group satisfies properties (i)-(iii) then $\chi(xy)=\chi(x)\chi(y)$ for every nonidentity  $p$-element $x\in G$, every nonidentity $p'$-element of prime power order $y\in G$  and every $\chi\in\Irr(G)$. In order to see this, it suffices to prove that $\chi(g)=0$ for any $g\in G-H$ and any nonlinear $\chi\in\Irr(G)$. 

Since $A$ and $C$ are abelian and commute, $G/D$ is abelian. Therefore, all the nonlinear characters of $G$ lie over nonprincipal characters of $D$. Let $\alpha\in\Irr(D)$ be such a nonprincipal character. Since the action of $A$ on $D$ is Frobenius, $I_G(\alpha)\leq H$ (by Theorem 6.34 of \cite{isa}, for instance). It follows from Clifford's correspondence, that every $\chi\in\Irr(G)$ that lies over $\alpha$ is induced from some character of $H$. Therefore, $\chi$ vanishes on $G-H$, as desired.
\end{proof}

\section{Further results and questions}

There are a number of related questions that would be interesting to solve. The first  of them is immediately motivated by the results in this paper:

\begin{que} 
Which are the (solvable) finite groups that satisfy the hypothesis of Theorem B?
\end{que}

According to some GAP computations with the SmallGroups library (we thank A. S\'aez for doing this program) it could be true that $G/F(G)$ is metabelian. However, it does not seem easy to obtain a classification of these groups. 

We have seen in theorems \ref{a}, \ref{st} and \ref{stchar} that we can weaken the hypothesis on $x$ and $y$ that appears in the statement of the results in the introduction by assuming that $x$ is a $p$-element, for some fixed prime $p$. It is interesting to consider what happens if we also assume that $y$ is a $q$-element, for some fixed prime $q\neq p$. For instance, we have the following.

\begin{thm}
Let $p\neq q$ two odd primes and let $G$ be a finite group. Let $\pi=\{2,p,q\}$. Assume that $(xy)^G=x^Gy^G$ for every $x, y\in G$ $\pi$-elements of prime power order with $(o(x),o(y))=1$. Then $G$ does not have any composition factors of order divisible by $pq$.
\end{thm}

This result can be proved as Theorem \ref{st}  using Conjecture 6.2 of \cite{gt} (which has been proven in \cite{lee}) instead of Theorem \ref{gt}. It is likely that we do not need to assume that $2\in\pi$. This should be related to  the Arad-Herzog conjecture (see \cite{gmt}). 
Notice that this hypothesis does not imply that $G$ is either $p$-solvable or $q$-solvable: take $G=S_1\times S_2$ with $S_1$ simple of order divisible by $p$ but not by $q$ and $S_2$ simple of order divisible by $q$ but not by $p$.

By analogy with the multiplicative functions in analytic number theorem, we will say that a character $\chi$ of a finite group $G$ is a multiplicative if $\chi(xy)=\chi(x)\chi(y)$ for every nonidentity $x,y\in G$ with $(o(x),o(y))=1$. (Sometimes, the term multiplicative character is used for the linear characters, which is a more complete analogy, but this is not convenient for our purposes.) Multiplicative characters seem rare in nonnilpotent groups.

\begin{que}
Which are the nonnilpotent groups that possess a faithful (nonlinear) multiplicative character?
\end{que}

Gagola groups (i.e. groups with a character that vanishes on all but two conjugacy classes; see \cite{gag}) are examples of such groups. Notice that among them there are nonsolvable groups. We are not aware of any example of a nonnilpotent group with a faithful multiplicative character that does not vanish off of some normal $p$-subgroup for some prime $p$. In particular, we believe that such characters do not exist in nonabelian simple groups.


\begin{thebibliography}{131}





%\bibitem{ah}
%\textsc{Z. Arad, M. Herzog},
%``Products of conjugacy classes in groups",
%Springer, 
%Berlin,
%1985.

\bibitem{bw} {\sc B. Baumslag, J. Wiegold}, {A sufficient condition for nilpotency in a finite group}. arXiv:1411.2877v1.



\bibitem{dy} {\sc E. C. Dade, M. K.  Yadav}, {Finite groups with many product conjugacy classes}. Israel J. Math.  \textbf{154} (2006), 29--49.


\bibitem{dghp} {\sc S. Dolfi, R. M. Guralnick, M. Herzog, C. E. Praeger}, {A new solvability criterion for finite groups}. J. London Math. Soc.  \textbf{85} (2012), 269--281.

\bibitem{gag} {\sc S. Gagola}, {Characters vanishing on all but two conjugacy classes}. Pacific J. Math. \textbf{109} (1983), 363--385.

%\bibitem{gl} {\sc S. Gagola, M. L. Lewis}, {A character-theoretic condition characterizing nilpotent groups}. Comm. Algebra \textbf{27} (1999), %1053--1056.

\bibitem{gmt} {\sc R. M. Guralnick, G. Malle,  Pham Huu Tiep}, {Products of conjugacy classes in finite and algebraic simple groups}. Adv. Math.  \textbf{234} (2013), 618--652.

\bibitem{gn} {\sc R. M. Guralnick, G. Navarro}, {Squaring a conjugacy class and cosets of normal subgroups}. Proc. Amer. Math. Soc.  \textbf{144} (2016), 1939--1945.


\bibitem{gt} {\sc R. M. Guralnick, Pham Huu Tiep}, {Lifting in Frattini covers and a characteration of finite solvable groups}. Journal fur die Reine und Angewandte Mathematik \textbf{708} (2015), 49--72.

%\bibitem{hup}
%\textsc{B.~Huppert},
%``Character Theory of Finite Groups",
%de Gruyter,
%Berlin/New York,
%1998.


\bibitem{isa}
\textsc{I.M.~Isaacs},
``Character Theory of Finite Groups",
Dover,
New York,
1994.

\bibitem{hi} {\sc R. B. Howlett, I. M. Isaacs}, On groups of central type. Math. Z. {\bf 179} (1982), 555--569.


\bibitem{lad} {\sc F. Ladisch}, {Groups with anticentral elements}. {Comm.  Algebra} \textbf{36} (2008), 413--423.

\bibitem{lee} {\sc H. Lee}, {Triples in Finite Groups and a Conjecture of Guralnick and Tiep}. PhD thesis. University of Arizona.  2017. 

\bibitem{ms} {\sc  A. Moret\'o, A. Saez}, {Prime divisors of orders of products}. to appear in Proc. Edim. Math. Soc. 


\bibitem{suz}
{\sc M. Suzuki}, ``Group Theory II" 
 Springer-Verlag, New York,  1986.

\bibitem{tho} {\sc J.G. Thompson}, {Nonsolvable finite groups all of whose local subgroups are solvable}. { Bull. Amer. Math. Soc.} \textbf{74} (1968) 383--437.







\end{thebibliography}
\end{document}